\documentclass[reqno,a4paper,usenames,dvipsnames]{amsart}
\usepackage{amsmath, amsthm, amsfonts, amssymb, amscd, cancel, graphicx,mathrsfs, mathtools, pinlabel, soul, stmaryrd}

\usepackage[T1]{fontenc}
\usepackage[utf8]{inputenc}
\usepackage{palatino}

\usepackage{tikz}
\usepackage{array,diagbox}
\usepackage{textcomp}
\usepackage{overpic}
\usepackage{pgf}
\usepackage{pgfkeys}
\usepackage{pgfplots}
\pgfplotsset{compat=1.18}
\usepgfplotslibrary{units}
\usepackage{graphicx} % Required for inserting images
\usepackage{caption}
\usepackage{subcaption}
\usepackage{thmtools}
\usepackage{thm-restate}

\theoremstyle{definition}
\newtheorem{thm}{Theorem}%[section]
\newtheorem{defn}[thm]{Definition}
\newtheorem{prop}[thm]{Proposition}
\newtheorem{cor}[thm]{Corollary}

\newtheorem*{quest*}{Question}
\newtheorem{rem}[thm]{Remark}
\newtheorem{lemma}[thm]{Lemma}

\newtheorem*{conj*}{Conjecture}
\newtheorem*{probl*}{Problem}

\newcommand{\D}{\mathbb{D}}

\newcommand{\tw}{\mathrm{tw}}
\newcommand{\gd}{g_\mathrm{ds}}

\newcommand{\splitatcommas}[1]{%
  \begingroup
  \begingroup\lccode`~=`, \lowercase{\endgroup
    \edef~{\mathchar\the\mathcode`, \penalty0 \noexpand\hspace{0pt plus 1em}}%
  }\mathcode`,="8000 #1%
  \endgroup
}

\usepackage{hyperref}
\usepackage{xcolor}
\hypersetup{
	colorlinks=true,
	linkcolor={magenta},
	citecolor={blue},
	urlcolor={black}
}

\usepackage{cleveref}
\usepackage[backend=bibtex,maxbibnames=99,giveninits=true,sorting=nty,style=alphabetic%numeric
]{biblatex}
\AtBeginBibliography{\footnotesize}
\DeclareFieldFormat{postnote}{#1}
\title[Twisting number is bounded below by doubly slice genus]{%On the twisting number of ribbon knots 
The twisting number
of a ribbon
knot is bounded below by its doubly slice genus}
\author{Vitalijs Brejevs}
\address{Department of Mathematics, University of Vienna, Oskar-Morgenstern-Platz 1, 1090 Vienna, Austria}
\email{vitalijs.brejevs@univie.ac.at}
\author{Peter Feller}
\address{Department of Mathematics, ETH Z\"urich, R\"amistrasse 101, 8092 Z\"urich, Switzerland}
\email{peter.feller@math.ch}
\begin{document}
\begin{abstract}
    The twisting number of a %ribbon
    knot $K$ is the minimal number of tangle replacements on the symmetry axis of $J \# -J$ for any knot $J$ that is required to produce a symmetric union diagram of $K$.
    We prove that the twisting number is bounded below by the doubly slice genus and produce examples of ribbon knots with arbitrarily high twisting number, addressing a problem of Tanaka. As an application, we determine hitherto unknown doubly slice genera of some knots with 12 crossings.
    
    % To this end, we establish a result of independent interest that the doubly slice genus of $K$ is bounded above by the oriented band move distance between $K$ and any weakly doubly slice link, which also enables us to determine hitherto unknown doubly slice genera of some knots with 12 crossings.
\end{abstract}
\maketitle
\vspace{-2em}
\section{Introduction}

We work in the smooth category throughout and all our manifolds are oriented. A knot $K$ is called a \emph{symmetric union knot}, or simply a \emph{symmetric union}, if it can be obtained by taking the connected sum $J \# -J$ of a knot $J$ 
%, called the \emph{partial knot} of $K$, 
with its mirror $-J$, and replacing trivial 4-tangles that intersect the symmetry axis of $J \# -J$ by 4-tangles twisted along the axis, called \emph{twist regions}; see Figure~\ref{fig:su}. This construction, originally proposed by Kinoshita and Terasaka in~\cite{kt:original} and developed by Lamm in~\cite{lamm:original}, determines a ribbon-immersed disc in $S^3$ bounded by $K$ that is swept out by an arc between points lying symmetrically across the axis. Symmetric unions are ubiquitous among small ribbon knots: at least 144 out of 159 ribbon knots with at most 12 crossings admit symmetric union diagrams~\cite{lamm:nonsym}; moreover, so do all ribbon 2-bridge knots~\cite{lamm:2bridge}. This prompts the following open conjecture that is central to the study of symmetric unions.

\begin{conj*}[`Ribbon--symmetric union conjecture']
\label{conj:rsu}
    Every ribbon knot is a symmetric union.
\end{conj*}

This conjecture is challenging, in particular, because the number of twist regions required to realise a ribbon knot $K$ as a symmetric union may be arbitrarily large. The smallest such number is called the \emph{twisting number}\footnote{The twisting number as defined here is also called the \emph{minimal twisting number} in the literature.} of $K$ and denoted $\tw(K)$. We set $\tw(K) = \infty$ if $K$ is not a symmetric union; in other words, a ribbon knot with $\tw(K) = \infty$ would provide a counterexample to the conjecture. 
By the construction of symmetric unions, $\tw(J \# -J) = 0$ for any knot $J$ and $\tw(K) \geq 1$ for any prime knot $K$. Twisting number one knots have been studied extensively by Tanaka~\cite{tanaka:jones,tanaka:composite,tanaka:satellite} and K\"ose~\cite{kose:amphichirality,kose:thesis}; in fact, in~\cite{tanaka:jones}, Tanaka asked whether more than one twist region is ever necessary to realise a knot as a symmetric union.

\begin{probl*}
\label{q:tanaka}
    Does there exist a symmetric union knot $K$ with $\tw(K) \geq 2$?
\end{probl*}

Via a factorisation condition for the Jones polynomial of a knot with $\tw(K) = 1$, it was shown in~\cite{tanaka:jones} that indeed there exist knots with $\tw(K) = 2$. However, this prompts the natural question whether there are examples of knots with $\tw(K)\geq 3$.
To exhibit knots with $\tw(K)=N$ for all $N\geq 3$, we establish
a lower bound on $\tw(K)$ in terms of the {doubly slice genus} $\gd(K)$.

\begin{thm}
\label{thm:main} For all knots $K$, we have
    $\gd(K) \leq \tw(K)$.
\end{thm}
Here, as first defined and studied in~\cite{livingstonmeier:ds}, the \emph{doubly slice genus} $\gd(K)$ of a knot $K$, and, more generally, a link, is the minimal genus among unknotted closed surfaces $F \subset S^4$ such that $K=S^3\pitchfork F$, where $S^3$ is identified with the equatorial sphere in $S^4$.

The key input for the proof of Theorem~\ref{thm:main} is the following bound on the doubly slice genus, which we discuss in detail in Section~\ref{sec:mainproof}. It is a result of independent interest that may be viewed as an extension of the theorem of McDonald from~\cite{mcdonald:bandnumber} that the {band unknotting number} of a knot $K$ is at least $\gd(K)$.

\begin{restatable}{prop}{boundonds}
\label{prop:boundonds}
    Let $L$ be a weakly doubly slice link, i.e., $\gd(L)=0$. If $L$
    can be turned into a link $L'$ by $t$ oriented band moves, then $\gd(L')\leq t$.
\end{restatable}

Orson and Powell have established in~\cite{orsonpowell:signatures} that,
for all knots $K$,
\[
    \label{eq:sigbound}
    \underset{\omega \in S^1 \setminus \{1\}}{\mathrm{max}} | \sigma_\omega(K)| \leq \gd(K),
    \tag{$\dagger$}
\]
where $\sigma_\omega(K)$ is the Levine--Tristram signature of $K$ corresponding to $\omega \in S^1 \subset \mathbb{C}$. The bound~\eqref{eq:sigbound} is a refinement of a classical bound for the 4-genus $g_4(K)$ of $K$ given by $|\sigma_\omega(K)| \leq 2 g_4(K) \leq \gd(K)$ for any $\omega$ such that $\Delta_K(\omega) \neq 0$, where $\Delta_K$ is the Alexander polynomial of $K$. 
The inequality~\eqref{eq:sigbound} coupled with Theorem~\ref{thm:main} enables
us to show that there are ribbon
knots with arbitrarily high finite twisting number.\footnote{In fact, the authors' original approach to finding ribbon knots with large twisting numbers was to establish $|\sigma_\omega(K)| \leq \tw(K)$ by purely $3$-dimensional arguments.}

\begin{cor}
\label{cor:820}
    Let $K = 8_{20}$ and fix $N > 0$. Then $\tw(\#^N K) = \gd(\#^N K) = N$. 
\end{cor}
\begin{proof}
    One may calculate $\sigma_\omega(K) = 1$ for $\omega = e^{\frac{\pi i}{3}}$ (cf.~KnotInfo~\cite{knotinfo}).
    By additivity of Levine--Tristram signatures, inequality~\eqref{eq:sigbound}, and Theorem~\ref{thm:main}, we have
    $N = |\sigma_\omega(\#^N K)| \leq \gd(\#^N K) \leq \tw(\#^N K)$. On the other hand, using the one-twist symmetric union diagram for $K$ from~\cite{lamm:original} we can construct an $N$-twist diagram for $\#^N K$ by taking connected sums along the symmetry axis, hence $\tw(\#^N K) = N$.
\end{proof}

\begin{rem}
An analogous argument shows that $\tw(\#^N K)\geq N$ for any prime ribbon knot $K$ with non-vanishing Levine--Tristram signature at some root of $\Delta_K$, e.g., $10_{140}$.
Moreover, Theorem~\ref{thm:main} enables us to determine $\gd(K)$ for $K \in \{ 12n_{553},\allowbreak 12n_{556} \}$ that has been priorly only known to lie in the interval $[1, 2]$ (see~\cite{karageorghis:dsgenera}): since $K$ admits a symmetric union diagram with one twist region by~\cite[Appendix]{lamm:nonsym}, we conclude that $\gd(K) = \tw(K) = 1$.

% Moreover, using Theorem~\ref{thm:main} we can determine not only $\tw(K)$, but even $\gd(K)$ for $K \in \{ 12n_{553}, 12n_{556} \}$. It has been priorly known that $\gd(K) \in [1, 2]$ (see~\cite{karageorghis:dsgenera}); however, since $K$ admits a symmetric union diagram with one twist region by~\cite[Appendix]{lamm:nonsym}, we conclude that $\gd(K) = \tw(K) = 1$.
% admits a symmetric union diagram with one twist region by~\cite[Appendix]{lamm:nonsym}, so $\gd(K) = \tw(K) = 1$. Note that $\gd(K)$ appears to have been priorly only known to be either one or two; see~\cite{karageorghis:dsgenera}.    
\end{rem}

\begin{rem}
    Since the initial publication we were informed by William Rushworth that Proposition~\ref{prop:boundonds} follows from Corollary~14 in~\cite{bodenetal:ribbonsurfaces} by Boden, Elmacioglu, Guha, Karimi, Rushworth, Tang and Wang Peng Jun. Their proof relies on the observation that attaching 1-handles to a trivial $S^2 \subset S^4$ corresponding to the band moves on $L$ always yields an unknotted surface~\cite[Proposition~11.2]{kamada:book}. Our proof is more explicit as it keeps track of a specific handlebody that establishes unknottedness of the relevant surface in $S^4$. We expect it to generalise to the case when $L$ is the boundary of the cross-section of the equatorial $S^3 \subset S^4$ by any embedded closed connected 3-manifold rather than just $B^3 \subset S^4$.
\end{rem}

\subsection*{Organisation of the paper} In Section~\ref{sec:sudef} we define symmetric unions and provide necessary background on doubly slice knots and links. The proofs of Theorem~\ref{thm:main} and Proposition~\ref{prop:boundonds} occupy Section~\ref{sec:mainproof}.

\subsection*{Acknowledgements.} VB is supported by the Austrian Science Fund (FWF) grant `Cut and Paste Methods in Low Dimensional Topology'. PF gratefully acknowledges the support by the SNSF grant 181199.
Both authors are thankful to the winter school \emph{Winter Braids XIII}, where this project was initiated, as well as to Feride Ceren K\"ose, Charles Livingston and William Rushworth for helpful comments on the initial draft.

\section{Symmetric unions and weakly doubly slice links}
\label{sec:sudef}

The purpose of this section is to recall the definitions of symmetric unions and weakly doubly slice links, as well as to establish notation for the proof of Theorem~\ref{thm:main}.
We identify isotopic knots, links and tangles.%\pf{added this, claim it makes clearer what follows below}

\subsection{Symmetric unions}
\label{subsec:sudef}

Fix an unoriented diagram $D_J$ of a knot $J$ and reflect it about an axis $A$ in $\mathbb{R}^2 $ that does not intersect $D_J$ to get a diagram $-D_J$ of $-J$. Then, for some $k \geq 0$ choose a collection of $k+1$ disjoint embedded discs $T_0, T_1, \dots, T_k$ in $\mathbb{R}^2$ such that each $T_i$ is preserved under reflection about $A$ and intersects the disjoint union $D_J \sqcup -D_J$ in two crossingless arcs. Fix an integer $\mu$ with $1 \leq \mu \leq k+1$. Treating each $T_i$ as a rational 0-tangle, replace it by an $m_i$-tangle $T_i'$, %\footnote{$m_i$-tangle never defined, I think needs a pic}
where $m_i = \infty$ for $0 \leq i \leq \mu-1$ and $m_i \in \mathbb{Z} \setminus \{0\}$ for $\mu \leq i \leq k$. Choose an orientation on the resulting diagram and denote it by $D_L = (D_J \sqcup -D_J)(\infty_\mu, n_1, \dots, n_l)$, where $l = k - \mu + 1$ and $n_i = m_{\mu -1 + i}$, and call the tangles $T_{\mu}', \dots, T_k'$ \emph{twist regions}. The construction is illustrated in Figure~\ref{fig:su}.
\begin{figure}[h]
\centering
\begin{overpic}
    [scale=.18]{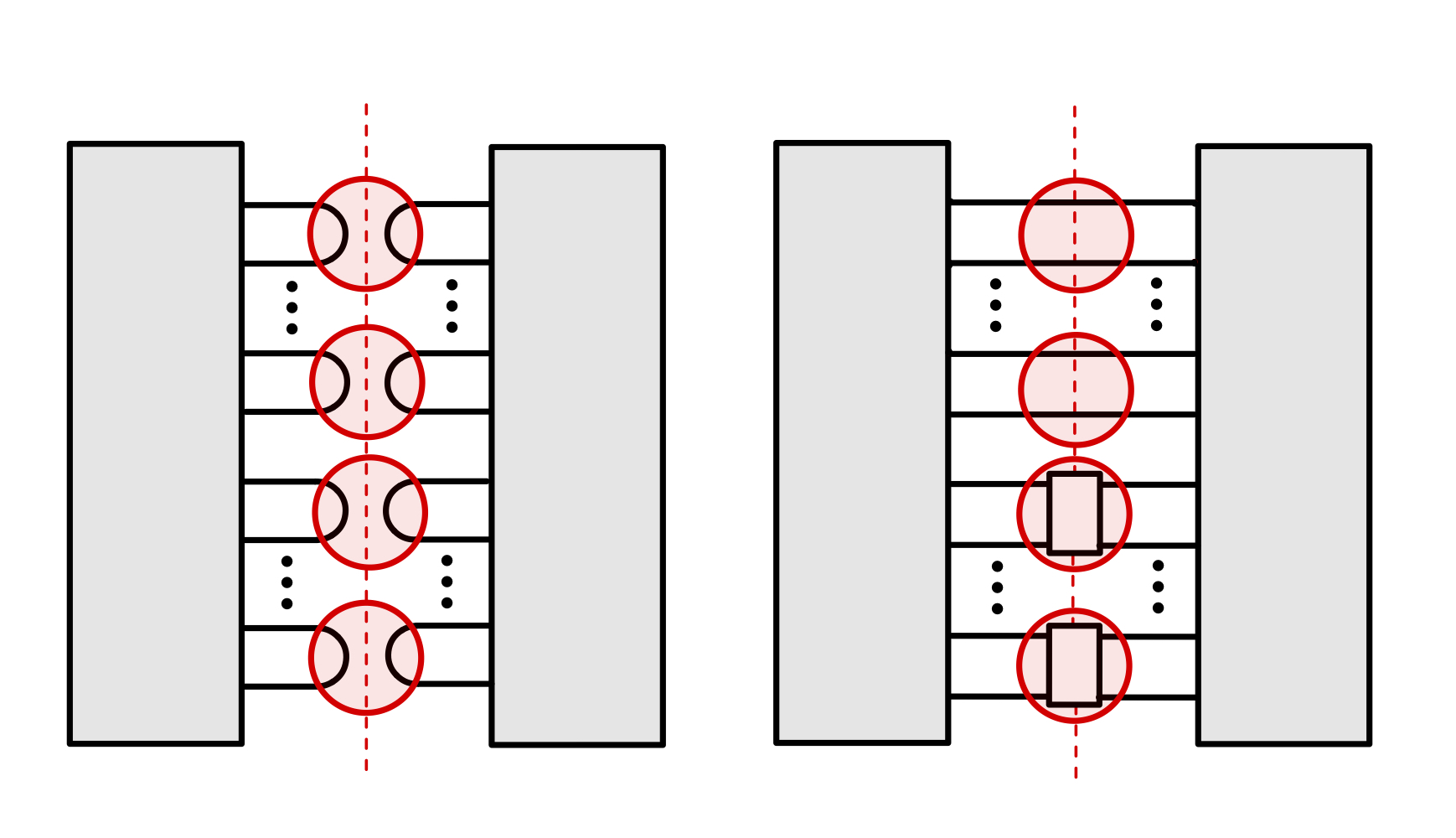}
    \put(21,46){$\scriptstyle T_0$}
    \put(21,35.8){$\scriptstyle T_{\mu-1}$}
    \put(20.5,26.4){$\scriptstyle T_{\mu}$}
    \put(21,16.5){$\scriptstyle T_{k}$}
    \put(70,46){$\scriptstyle T_0'$}
    \put(70,35.6){$\scriptstyle T_{\mu-1}'$}
    \put(68.5,26.1){$\scriptstyle T_{\mu}'$}
    \put(69.5,16.4){$\scriptstyle T_{k}'$}
    \put(72.5,21.7){$\scriptstyle n_1$}
    \put(72.5,11.4){$\scriptstyle n_l$}
    \put(8.5,25){$D_J$}
    \put(36.5,25){$-D_J$}
    \put(57,25){$D_J$}
    \put(85,25){$-D_J$}
    \put(26,47){$A$}
    \put(75,47){$A$}
    \put(47,25){$\longrightarrow$}
\end{overpic}
\vspace{-1em}
\caption{Construction of the symmetric union diagram $(D_J \sqcup -D_J)(\infty_\mu,\allowbreak n_1, \dots, n_l)$; rectangles labelled $n_1, \dots, n_l$ contain the respective number of `vertical' half-twists.}\label{fig:su}
\end{figure}
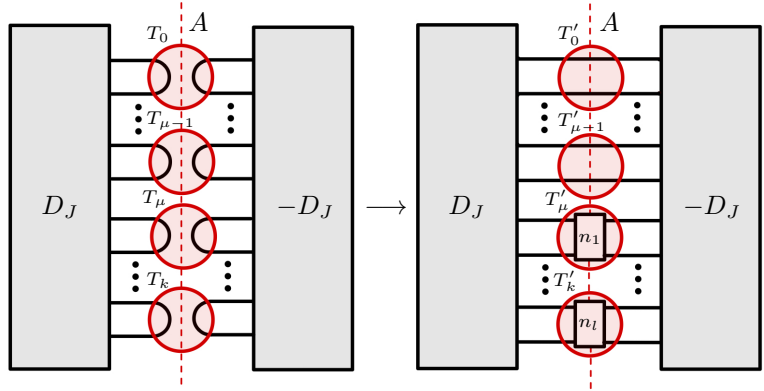

\begin{defn}
\label{def:sulink}
    \hfill
    \begin{enumerate}
        \item A link $L$ is a \emph{symmetric union} if it admits a diagram $D_L$ as above for some $J$, $\mu$, $l$ and $n_i \in \mathbb{Z} \setminus \{0\}$ for $1 \leq i \leq l$.
        \item The \emph{twisting number} $\tw(L)$ of $L$ is the minimal number of twist regions in any diagram $D_L$.
        \item The \emph{$\infty$-resolution} of $D_L$, denoted $L_\infty$, is the symmetric union represented by the diagram $(D_J \sqcup -D_J)(\infty_{\mu + l})$.
    \end{enumerate}
\end{defn}

For a symmetric union $L$, by~\cite[Remark~2.2]{lamm:original} and~\cite[Theorem~5.1]{lamm:original} we have the following:
    \begin{enumerate}
        \item the link $L$ has $\mu$ components, hence $L$ is a knot if and only if $\mu = 1$;
        \item upon assigning an orientation to $(D_J \sqcup -D_J)(\infty_\mu, n_1, \dots, n_k)$, 
        both strands at each axial crossing are oriented such that they cross $A$ in the same direction;
        \item each component of $L$ bounds an immersed disc in $S^3$ that only has intersections of ribbon type with itself and the discs bounded by other components of $L$, hence $L$ is a ribbon link.
    \end{enumerate}

%Recall that 
A \emph{band} $b$ for a link $L$ is the image of a smooth embedding $\phi\colon [0,1]\times[0,1]\to S^3$ such that $\phi$ maps $\{0,1\}\times [0,1]$ into $L$ while preserving the orientation (i.e., $\{0,1\}\times [0,1]$ carries the orientation induced from being part of the boundary of $[0,1]\times[0,1]$). By a \emph{band move} we mean embedded surgery of $L$ along $b$, namely a smoothing of
\[\left(L\setminus\phi(\{0,1\}\times [0,1])\right) \cup \phi([0,1]\times \{0,1\}); \]
we say that $b$ \emph{guides} the band move. In view of the observations~(1) and~(2) above, it is clear that a symmetric union $L$ with $l$ twist regions can be transformed into $L_\infty$ via $l$ band moves as shown in Figure%and a number of R1 moves 
~\ref{fig:band}. Note that the orientation of the resulting $L_\infty$ depends on the choice of orientation on $L$ and the parities of $n_1, \dots, n_l$.
\begin{figure}[ht]
\centering
\begin{overpic}
    [scale=.06]{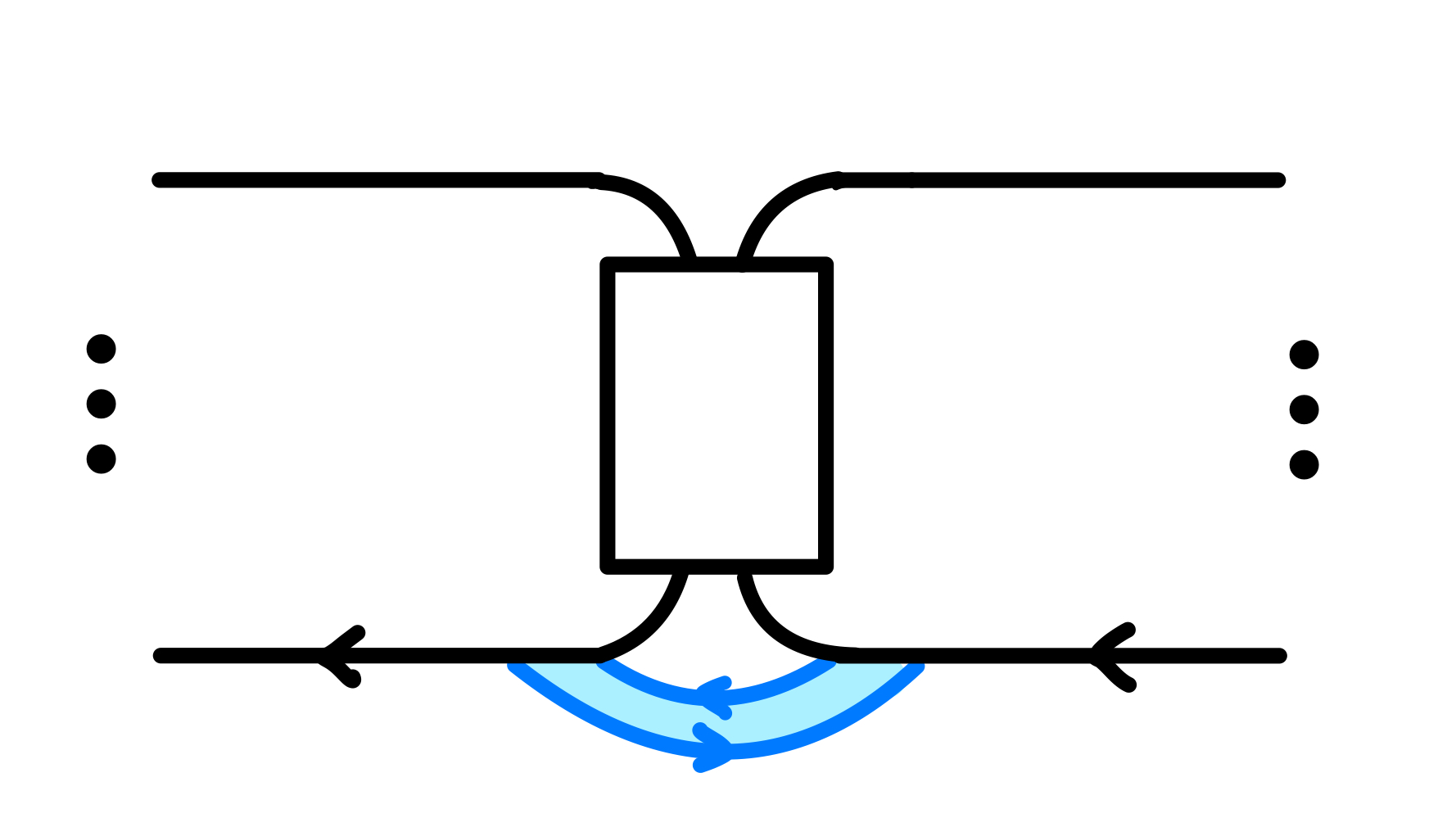}
    \put(45,25){$n_i$}
    \put(34,1){\color{blue}{$b$}}
\end{overpic}
\vspace{-0.5em}
\caption{The band move guided by the band $b$ (blue) %followed by $n_i$ applications of the R1 move
changes the $n_i$-tangle into the $\infty$-tangle. The orientation of the top strand depends on the parity of~$n_i$.}\label{fig:band}
\end{figure}

\subsection{Doubly slice links}

An oriented %$\mu$-component
link $L \subset S^3$ is \emph{weakly} %(resp., \emph{strongly}) 
\emph{doubly slice} if there exists an unknotted oriented 2-sphere $S \subset S^4$
%(resp., an oriented unlink of 2-spheres $ S_1 \sqcup \dots \sqcup S_m \subset S^4$)
such that for the equatorial $S^3 \subset S^4$ we have $S \pitchfork S^3 = L$%(resp., $S_i \pitchfork S^3 = L_i$ for all $1 \leq i \leq \mu$)
, inducing the given orientation on $L$. This extends the notion of double sliceness for knots to links: a knot $K$ is called \emph{doubly slice} if it bounds two slice discs, one in each hemisphere of $S^4$, that yield an unknotted $S^2 \subset S^4$ upon gluing along $K$, or, equivalently, if $K$ is weakly doubly slice. 
%For knots these notions coincide, hence a knot $K \subset S^3$ is \emph{doubly slice} if it bounds two slice discs, one in each hemisphere of $S^4$, that yield an unknotted $S^2 \subset S^4$ upon gluing along $K$.

A prototypical example of a doubly slice knot is the connected sum $J \# -J$ for any knot $J$: it occurs as a cross-section of the $\pm1$-twist spin of $J$ which is an unknotted 2-sphere by a classical theorem of Zeeman~\cite{zeeman:twisting}. A recent result of McCoy and McDonald from~\cite{mccoymcdonald:ds} can be viewed as a generalisation of this example to links; we give an equivalent restatement of it in terms of our objects of interest.

\begin{prop}[{\cite[Theorem~1.5]{mccoymcdonald:ds}}]
\label{prop:kinfty}
    Every $\infty$-resolution $L_\infty$ %of any symmetric union link $L$
    is weakly doubly slice.
    %for every choice of orientation on $L_\infty$.
\end{prop}

In particular, it follows that a symmetric union link $L$ with $\tw(L) = t$ is at most $t$ oriented band moves away from a weakly doubly slice link, namely $L_\infty$.

\section{Proof of Theorem~\ref{thm:main}}
\label{sec:mainproof}

Recall that the main input for the Theorem~\ref{thm:main} is Proposition~\ref{prop:boundonds}, which states that performing $t$ band moves on a weakly doubly slice link yields a link with doubly slice genus at most $t$.

\boundonds*

The idea of the proof is to consider a $3$-ball $B\subset S^4$ such that $L = \partial B\pitchfork S^3$ and perform $t$ ambient $1$-handle attachments to $B$ whose cores are given by the cores of the bands that guide the $t$ band moves, which then guarantees that the genus~$g$ handlebody $H$ resulting from attaching the 1-handles satisfies $L'=\partial H\pitchfork S^3$. The authors first saw this strategy implemented in~\cite{chen:ds}, in a slightly different context.

A crucial technical step in realising this strategy concerns arranging that $B\subset S^4$ misses the interiors of the cores of the bands; see Case~2 in the proof below. For this we need the following lemma, which is a slight generalisation of~\cite[Lemma~3.7]{chen:ds}.

\begin{lemma}
\label{lem:stablemma}
Let $B$ be a $3$-ball in $S^4$ with $F\coloneqq B\pitchfork S^3$, and let $F_a$ be the result of stabilising $F$ along an arc $a$. Then there exists an ambient isotopy carrying $B$ to a 3-ball $B'$ that fixes the boundary such that $F'\coloneqq B'\pitchfork S^3$ is given by $F' = F_\mathcal{A}\cup \partial(\mathrm{nb}(U))$,
where $F_\mathcal{A}$ is the result of stabilising $F$ along a system of arcs $\mathcal{A}$ that contains $a$, and $U \subset S^3$ is a link in the complement of $F\cup \mathcal{A}$.
Moreover, if the arc $a$ starts and ends on the same component of $F$, then $\mathcal{A}$ can be chosen to be~$\{a\}$.
\end{lemma}

We defer the proof of Lemma~\ref{lem:stablemma} to the end of this section, noting that in the case that $a$ starts and ends on the same component of $F$, the Lemma~\ref{lem:stablemma} states what Chen establishes in~\cite[Proof of Lemma 3.7]{chen:ds}.\footnote{To be very precise,
we note that~\cite[Lemma 3.7]{chen:ds} assumes that $\partial F = K$ is a {superslice} knot and that $\partial B$ is an unknotted 2-sphere that arises by doubling a slice disc for $K$; however, the argument given never makes use of these extra assumptions.} %, and the proof goes through line by line.

\begin{proof}[Proof of Proposition~\ref{prop:boundonds}]

Let $\mathcal{B} = \{ b_1,\dots, b_t \}$ be a set of disjoint bands that guide the $t$ band moves from $L$ to $L'$. Recall from Subsection~\ref{subsec:sudef} that the band $b_k$ is the image of a smooth embedding $\phi_k\colon [0,1]\times[0,1]\to S^3$ such that $\phi_k$ maps $\{0,1\}\times [0,1]$ into $L$ whilst preserving the orientation (i.e., $\{0,1\}\times [0,1]$ carries the orientation induced from being part of the boundary of $[0,1]\times[0,1]$).
%and the result of embedded surgery of $L$ along the $b_k$, i.e.~a smoothing of
%\[\left(L\setminus\bigcup_k \phi_k(\{0,1\}\times [0,1])\right)\bigcup_k\phi_k([0,1]\times \{0,1\}), \]
%is $K$. \vb{I put the definition of a band move in Section 2, could drop the details here}
Also, let $B$ be a smoothly embedded $3$-ball in $S^4$ that witnesses the double-sliceness of $L$, that is, $B$ and $\partial B$ intersect $S^3\subset S^4$ transversally and $\partial B\pitchfork S^3=L$.
We write $F\coloneqq B\pitchfork S^3$.\\

\noindent\textbf{Case 1: the bands $\mathcal{B}$ miss $F$.}
We first discuss the case when $F\cap b_k=L\cap b_k$ for all $b_k \in \mathcal{B}$; in other words, the case when $b_k$ only intersects $F$ in $\phi_k(\{0,1\}\times [0,1])$. In this case we build a genus $g$ handlebody $H \subset S^4$ by adding $1$-handles $h_k$ to $B$ as follows.

Let $\D$ denote the standard unit disc and identify a diameter of it with $[0,1]$. We define $h_k$ as the image of a smooth embedding $\Phi_k\colon\D\times [0,1]\to S^4$ such that $\Phi_k$ restricts to $\phi_k$ on $[0,1]\times [0,1]\subset \D\times [0,1]$. By construction, the handlebody $H$ resulting from attaching the $1$-handles $h_k$ to $B$ is a genus $t$ handlebody such that $\partial H\pitchfork S^3=L'$. This yields the desired result.\\

\noindent\textbf{Case 2: the bands $\mathcal{B}$ hit $F$.}
Assume now that $F$ intersects some $b_k$ away from $\phi_k(\{0,1\}\times [0,1])$. In the rest of the argument, we show that by an isotopy of $B$ in $S^4$ that fixes the boundary $\partial B$, we can arrange for $B\pitchfork S^3$ to only intersect $b_k$ in $\phi_k(\{0,1\}\times [0,1])$. Once this is achieved for every $b_k \in \mathcal{B}$, we argue as in Case~1.

We may assume that $F$ has only one non-closed component; in other words, we have $F=F_0\cup\dots\cup F_n$, where the components $F_l$ are connected oriented surfaces with $\partial F_0=L$ and $F_l$ closed for $l>0$. This can be achieved by iteratively applying the following procedure. First, stabilise a non-closed component of $F$ to another component along an arc $a$ to get a surface $F_a$. Then, apply Lemma~\ref{lem:stablemma} to isotop $B$ to a 3-ball $B'$ with $\partial B = \partial B'$ such that $B' \pitchfork S^3 = F_\mathcal{A} \cup \partial(\mathrm{nb}(U))$ with $a \in \mathcal{A}$ and $U$ some link in the complement of $F_\mathcal{A}$. Note that stabilisations never disconnect the components; thus, provided we avoid stabilising to closed tori coming from $\partial(\mathrm{nb}(U))$, this procedure eventually yields a surface with a unique non-closed component. We remark that this is the only time in the argument that we invoke Lemma~\ref{lem:stablemma} for stabilisation arcs that start and end on different components.

We may also 
%note that we can and do
assume that the core of each $b_k \in \mathcal{B} $ intersects $F_0$ algebraically zero times. This can be achieved by twisting $b_k$ around an arc that attaches it to $L$; see~Figure~\ref{fig:twist}.
%~\cite[Lemma~18]{fellerlewark:classical} for a detailed argument.
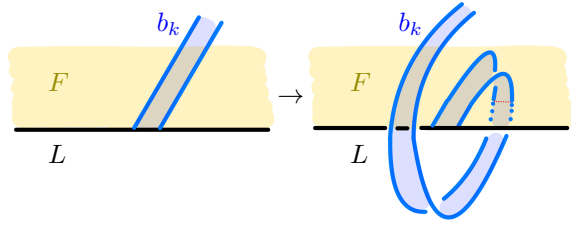
\begin{figure}[ht]
    \begin{overpic}
        [scale=.1]{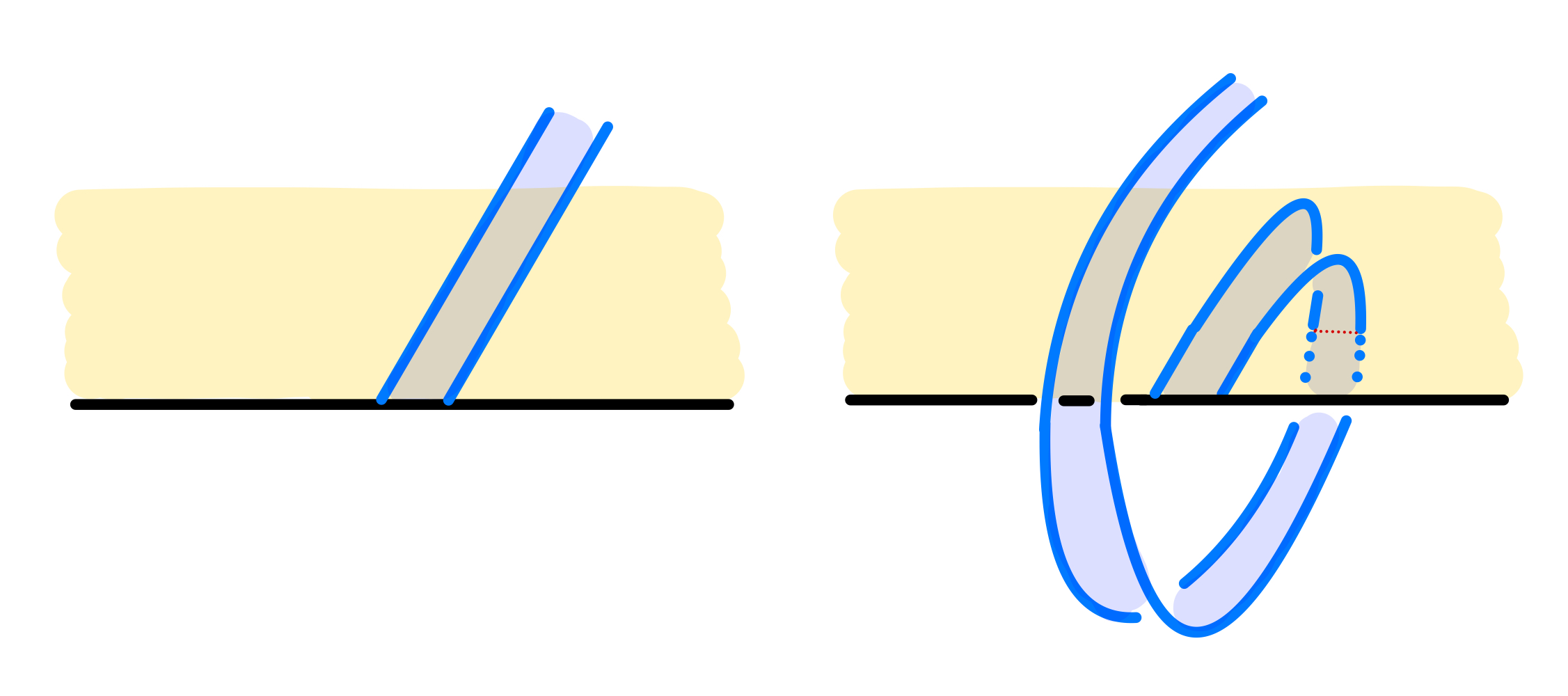}
        \put(10,13){$L$}
        \put(10,25){$\color{olive} F$}
        \put(28,35){$\color{blue} b_k$}
        \put(48,23){$\rightarrow$}
        \put(60,13){$L$}
        \put(60,25){$\color{olive} F$}
        \put(68,35){$\color{blue} b_k$}
    \end{overpic}
\vspace{-1em}
\caption{Modification of $b_k$ that changes the algebraic intersection of the core of $b_k$ with the non-closed component of $F$ without changing the isotopy class of $L'$.}
\label{fig:twist}
\end{figure}

Next, we stabilise $F$ to a surface $F'$ %with $\partial F = \partial F' = L$ 
that satisfies $F'\cap b_k=L\cap b_k$.
We do this by stabilisations that start and end at the same component of~$F$. Indeed, consider the core $I$ of~$b_k$. If~$b_k$ intersects any of the~$F_l$ for $l>0$, then one finds a subarc~$I' \subset I$ that starts and ends on the same $F_l$ and the interior of $I'$ misses~$F$. Stabilisation along $I'$ reduces the number of intersections of $I$ with $F_l$, and we may inductively arrange for $I$ to miss all $F_l$ with $l>0$ by such stabilisations. Similarly, since the algebraic intersection number between $F_0$ and $I$ is zero, if $I$ intersects $F_0$, then we can find a subarc $I' \subset I$ with endpoints on $F_0$ such that the interior of $I'$ is disjoint from $F$ and $I'$ approaches $F_0$ at the boundary from the same side (i.e., the intersection points of $I$ with $F_0$ that correspond to the endpoints of $I'$ have opposite sign).

%Once again,
By Lemma~\ref{lem:stablemma} we can find a boundary-fixing isotopy of $B$ to a 3-ball $B'$ such that $F'' \coloneqq B' \pitchfork S^3$ is the union of $F'$ and a closed surface that is the boundary of a link in $S^3$. Up to a small isotopy, $F''$ is disjoint from $b_k$. By sequentially applying the above procedure to each $b_k \in \mathcal{B}$, we may arrange that each $b_k$ intersects the resulting Seifert surface only in $\phi_k(\{0,1\}\times [0,1])$. Hence, we conclude the argument as in Case~1. \qedhere
\end{proof}

The proof of our main theorem now follows easily.

\begin{proof}[Proof of Theorem~\ref{thm:main}]
    Let $K$ be a symmetric union knot and let $(D_J \sqcup -D_J)(\infty_1,n_1,\allowbreak \dots, n_t)$ be a diagram for $K$ with $t = \tw(K)$. By Subsection~\ref{subsec:sudef}, the $\infty$-resolution $K_\infty$ of the diagram $(D_J \sqcup -D_J)(\infty_1,n_1,\allowbreak \dots, n_t)$ can be turned into $K$ via $t$ band moves.
    Since $K_\infty$ is weakly doubly slice %with any orientation
    by Proposition~\ref{prop:kinfty}, the result follows by Proposition~\ref{prop:boundonds}.
\end{proof}

It remains to prove Lemma~\ref{lem:stablemma}. We invite the reader to compare the proof we give below with the proof of~\cite[Lemma 3.7]{chen:ds}, where the objects of interest are given by more explicit parametrisations.

\begin{proof}[Proof of Lemma~\ref{lem:stablemma}]
Let $a$ be a stabilisation arc for $F$, i.e.,~$a$ is an arc in $S^3$ with endpoints $p_0$ and $p_1$ in $F\setminus \partial F$ and no further intersections with $F$. Furthermore, $a$ is transversal to $F$ and the two orientations of $S^3$ obtained by extending the orientation of $F$ with an outward pointing tangent vector of $a$ at $p_0$ and $p_1$, respectively, agree.

Let $b$ be an embedded arc in the interior of $B$ with endpoints $p_0$ and $p_1$ such that $b$ is transversal to $F$ and the two orientations of $B$ obtained by extending the orientation of $F$ with an outward pointing tangent vector of $b$ at $p_0$ and $p_1$, respectively, agree. Note that we allow for the possibility that $b\setminus\{p_0,p_1\}$ intersects $F\setminus\partial F$. However, if $p_0$ and $p_1$ lie on the same component of $F$, we choose $b$ such that $p_0$ and $p_1$ constitute the only intersections between $F$ and $b$, e.g., by choosing $b$ to lie in a collar of said component of $F$. Also note that $b$ intersects $S^3$ an even number of times, since, by our choice of how $B$ hits $F$ at its endpoints, $b$ points into the same hemisphere of $S^4$ at both ends.

The idea is to find a disc $D$ in $S^4$ with $ \partial D = a \cup b$ and to perform an isotopy of $B$ that is the identity outside a neigbourhood of $b$ and that moves $b$ along $D$; the disc $D$ will be a bigon (i.e., a smooth manifold with two corners that is homeomorphic to a disc). We first provide more details on $D$ and then explain the effect of the isotopy.

Let $W$ be a small %smooth
neighbourhood of $B\cup a$, and let $a'$ and $b'$ be push-offs (normal to $S^3$ and $B$) of $a$ and $b$, respectively, such that $a' \cup b'$ is a smooth curve $l\subset \partial W$.\footnote{Here, we are using that $a$ and $b$ both hit $F$ from the same side.} Let $A \subset W$ be the annulus with two corners that is traced out by the push-off of $a \cup b$ to $a' \cup b'$; the annulus $A$ shall serve as a collar of $\partial D$ in $W$.
See left-hand side of Figure~\ref{fig:D}.
\begin{figure}[ht]
\centering
\begin{subfigure}{.3\textwidth}
    \begin{overpic}
        [scale=.1]{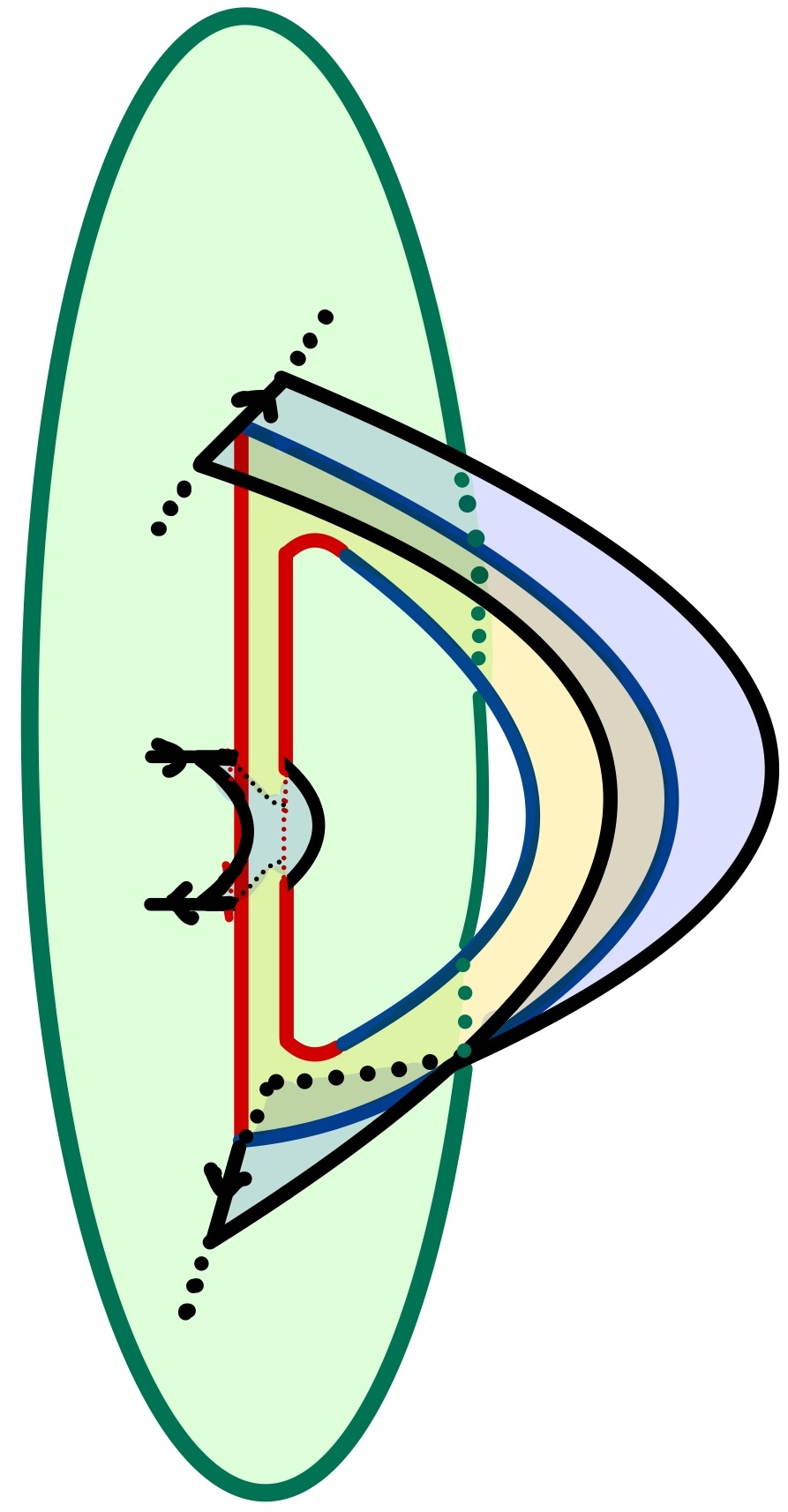}
        \put(46,46){$\scriptstyle \color{blue} a$}
        \put(36,46){$\scriptstyle \color{blue} a'$}
        \put(20,55){$\scriptstyle \color{red} b'$}
        \put(12,55){$\scriptstyle \color{red} b$}
        \put(21,35){$\scriptstyle \color{olive} A$}
        \put(14,86){$\scriptstyle B$}
    \end{overpic}
\end{subfigure}
\begin{subfigure}{.3\textwidth}
    \begin{overpic}
        [scale=.1]{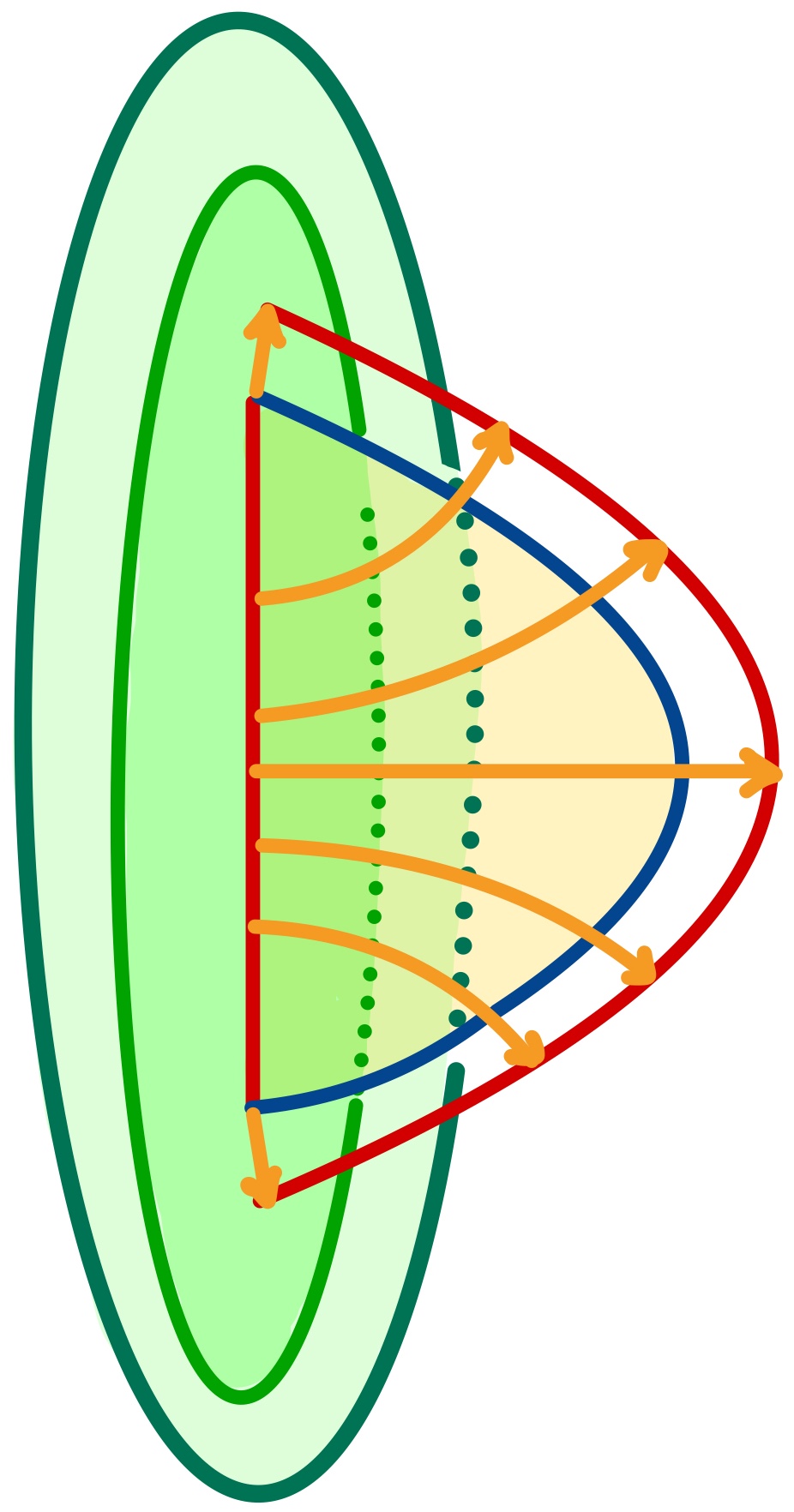}
        %\put(19,86){$\scriptstyle B$}
        \put(39,68){$\scriptstyle \color{red} \Psi(b)$}
        \put(38,51){$\scriptstyle \color{olive} D$}
    \end{overpic}
\end{subfigure}
\begin{subfigure}{.3\textwidth}
    \begin{overpic}
        [scale=.1]{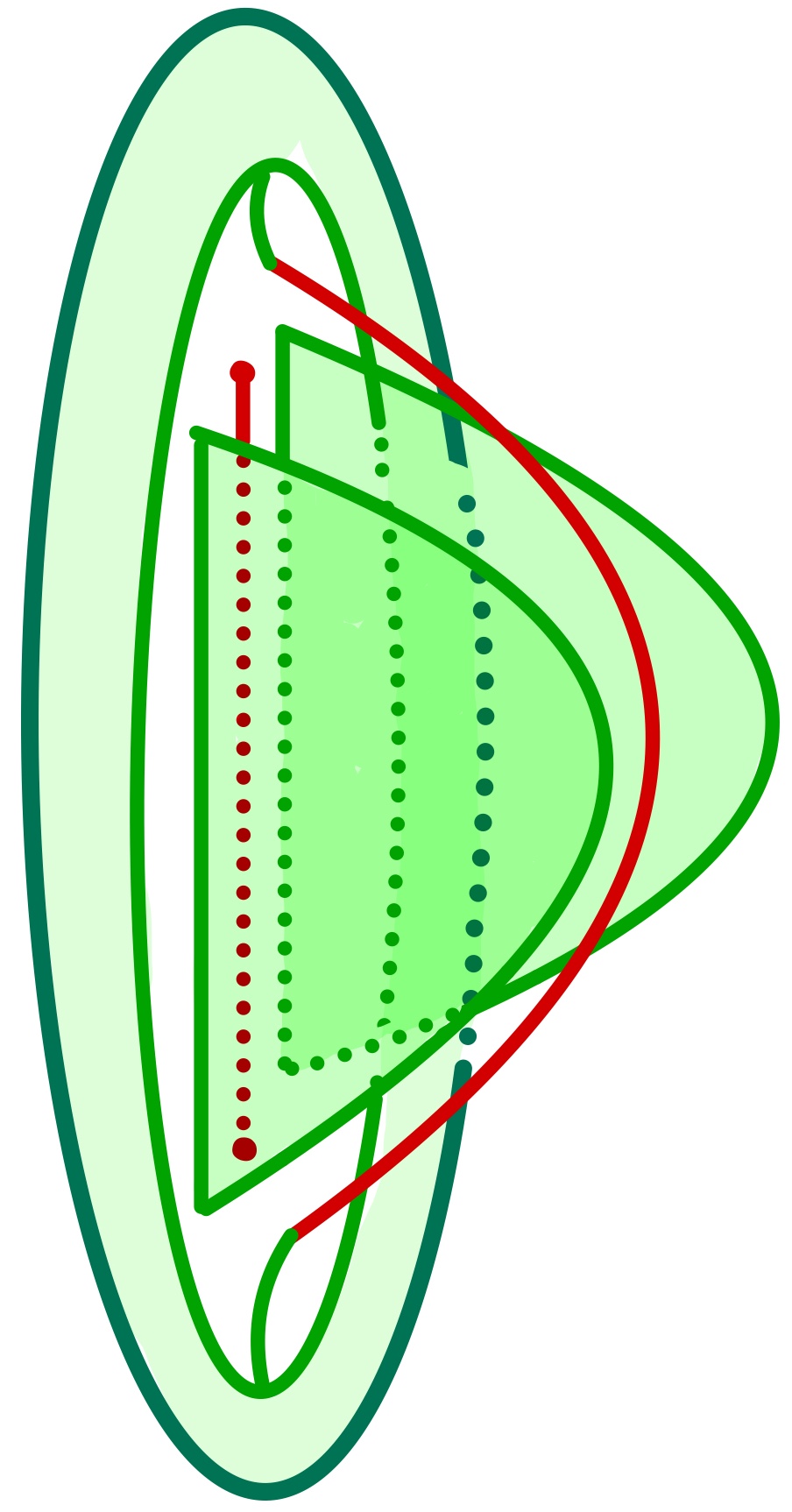}
        \put(33,73){$\scriptstyle \color{red} \Psi(b)$}
        \put(17,55){$\scriptstyle \color{red} b$}
    \end{overpic}
\end{subfigure}
\caption{Dimension reduced illustration of $D$, $B$ and $B'$.
\newline
Left: $B$ (green)%, drawn as a two-disc due to dimension reduction)
, part of $S^3$ (light blue), part of $F=B\cap S^3$ (black), $a$ and $a'$ (blue), $b$ and $b'$ (red), $A$ (yellow).
\newline
Middle: the bigon $D$ (yellow), the flow
$\Psi$ of $b$ under the vector field $v$ (orange arrows). %The support of the vector field $v$ is contained in a neighbourhood of $b$ (its intersection with $B$ is indicated darker green).
\newline
Right: $B'$, explicitly drawn are the subsets $B\cap B'$, $\phi(D\times \partial \D)$ (green%, $\partial \D$ is drawn as a $0$-sphere due to dimension reduction
), and $\Psi(b)$ (red).}
\label{fig:D}
\end{figure}

Noting that $W$ is diffeomorphic to $S^1\times D^3$ and using that every embedded $S^1$ in $S^4$ is the boundary of an embedded disc, we find an embedded $D'\subset S^4$ with $\partial D'=D'\cap W\subset \partial W$ that is transversal to $\partial W$ such that $\partial D'$ is a section of the normal projection of $\partial W$ to the core of $W$. Since $l$ and $\partial D'$ are isotopic in $\partial W\cong S^1\times S^2$ (e.g., by the 3-dimensional light bulb trick), we may assume (by an isotopy of $D'$) that $\partial D'=l$. We extend this to a smooth bigon with corners $D=D'\cup A$ that has $a\cup b$ as its boundary; see middle of Figure~\ref{fig:D}.

We further arrange for $D$ to intersect $S^3$ transversally. The intersection $D\pitchfork S^3$ consists of a union of circles $U$ in the interior of $D$, the arc $a$, and a (possibly empty) union $\bigcup_k^n I_k$ of properly embedded intervals with endpoints on $b$. Note that $a$ is the only interval in this intersection (i.e., $n=0$) if and only if $b$ is disjoint from $F$ away from $p_0$ and $p_1$.

The result $B'$ of an isotopy of $B$ that is the identity outside a neighbourhood of $D$ and pushes $b$ along $D$ slightly past $a$ is readily seen to have the following intersection with $S^3$. %effect on $F$.
The surface $F'\coloneqq B'\pitchfork S^3$ is a stabilisation of $F$ along the arcs $a$, $I_1$, $\dots$, $I_n\subset S^3$ together with the boundary of a small neighbourhood of the link $U\subset S^3$. %In the next paragraph, we provide a more detailed argument.
To see this more explicitly, we pick a vector field $v$ that is non-zero in a small neighbourhood of $D$ as follows. On $D$, the vector field $v$ is tangential to $D$ and normal to $b$, and its time one map $\Psi\colon S^4\to S^4$ maps $b$ slightly past $a$; see middle of Figure~\ref{fig:D}. Since the collar $A$ of $D$ is normal to $B$ and since $D$ intersects $S^3$ transversally, we find an embedding $\phi\colon D\times \D\to S^4$ with image $N(D)$ such that $\phi|_{a\times \D}$ (resp., $\phi|_{b\times \D}$) parametrises a cylinder in $S^3$ (resp., $B$) with the core $a$ (resp., $b$); in other words, up to smoothing corners, the trivialisation of the embedded normal bundle of $D$ induces a trivialisation of the embedded normal bundle of $a\cup b$ in $S^3\cup B$.
Furthermore, by transversality, we may arrange that
\[
S^3\cap N(D)= \phi((S^3\cap D)\times \D)=Z_a\cup \bigcup_{k=1}^n
Z_k
\cup \mathrm{nb}(U)
,\]
where $Z_a\coloneqq \phi(a\times \D)$ and $Z_k\coloneqq \phi(I_k\times \D)$ are cylinders in $S^3$ with the cores $a$ and $I_k$, respectively, and $\mathrm{nb}(U)
\coloneqq \phi(U \times \D)$ is a neighbourhood of $U$ in $S^3$; see Figure~\ref{fig:N(D)}.
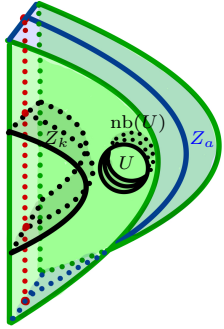
\begin{figure}[ht]
    \begin{overpic}
        [scale=.11]{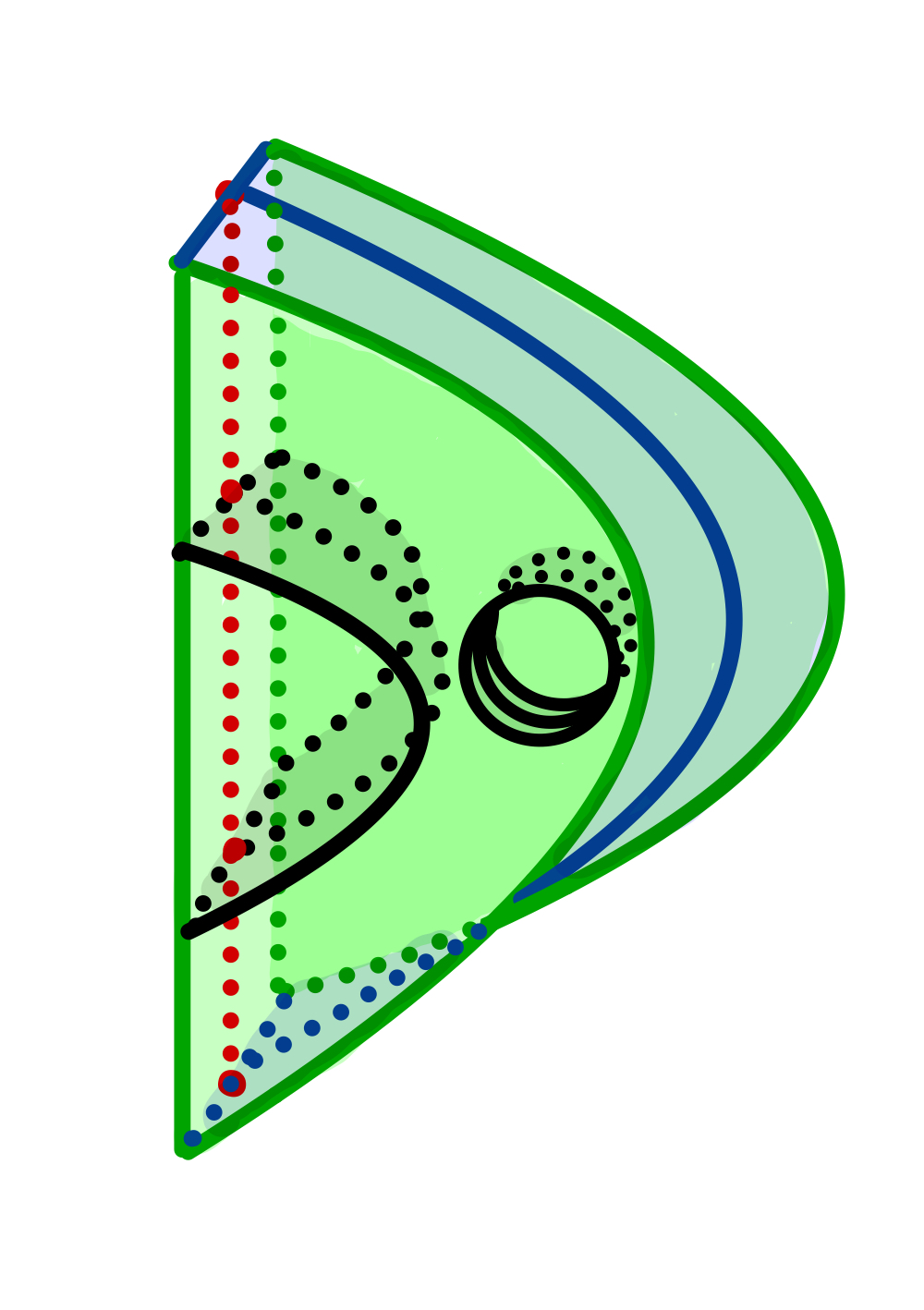}
        \put(57,55){$\scriptstyle \color{blue} Z_a$}
        \put(22,55){$\scriptstyle Z_k$}
        \put(38,59){$\scriptstyle \mathrm{nb}(U)$}
        \put(40,49){$\scriptstyle U$}
    \end{overpic}
\vspace{-1em}
\caption{Dimension reduced illustration of $N(D)$ and
\newline
$S^3\cap N(D)=\phi(D\cap S^3\times\D)=Z_a\cup\bigcup_{k=1}^nZ_k\cup\mathrm{nb(U)}$.}
\label{fig:N(D)}
\end{figure}

With this notation, one may describe the vector field $v$ in terms of the coordinates given by $\phi^{-1}$ or a slight extension of it with the domain containing the whole neighbourhood of $D$.
A natural choice is to set $v(x,y)=v(x,0)$ for all $x\in D$ and $y\in \D$ in coordinates given by $\phi^{-1}$, or, in other words, to extend $v$ normally from how it is chosen on $D$, and then quickly let it become zero outside of $\phi(B\times\D)$. However, to be able to explicitly write $F'$, we modify $v$ slightly such that $B'\coloneqq \Psi(B),$ where $\Psi$ is the time one flow of $v$, satisfies
$B'\cap N(D)=\phi(D\times \partial \D)$%(up to smoothing corners close to $\phi(a\times\partial \D$))
; see right-hand side of Figure~\ref{fig:D}.
With this choice of $v$, we can describe $F'\coloneqq \partial B'\cap S^3$ explicitly (up to smoothing corners afterwards) as 
\[\left(F\setminus
 \left(B_a\cup \bigcup_{k=1}^n
B_k\right)\right)\cup M_a \cup\bigcup_{k=1}^n
M_k
\cup \partial (\mathrm{nb}(U)),\]
where $B_a\coloneqq \phi(\partial a\times \D)$ and $B_k\coloneqq \phi(\partial I_k\times \D)$ are the bases of the cylinders $Z_a$ and $Z_k$, respectively, and 
$M_a\coloneqq \phi(a\times \partial \D)$ and $M_k\coloneqq \phi(I_k\times \partial \D)$ are the mantles of $Z_a$ and $Z_k$, respectively.
Hence, indeed, $F'$ is the result of stabilising $F$ along $a$ and the arcs $I_k$ union a closed surface given by the neighbourhood of a link $U\subset S^3$.
\end{proof}
%\newpage
\printbibliography
\end{document}